\documentclass[12pt]{amsart}

\usepackage{graphicx}
\usepackage{latexsym}
\usepackage{amsfonts}
\usepackage{amscd}
\usepackage{amssymb}
\usepackage{amsmath}
\usepackage{amsthm}
\usepackage{bm}
\usepackage[all]{xy}
\usepackage[lmargin=1.25in,rmargin=1.25in,tmargin=1.25in,bmargin=1in,dvips]{geometry}
\usepackage{pinlabel}
\usepackage{paralist}
\usepackage{mathtools}
\usepackage{enumitem}
\usepackage{tikz}
\usepackage{subfigure}
\usepackage[bookmarksnumbered=true, backref=true, bookmarksdepth=3]{hyperref}
\usepackage[a-1b]{pdfx}

\newtheorem{theorem}{Theorem}

\newtheorem{corollary}[theorem]{Corollary}
\newtheorem{proposition}[theorem]{Proposition}

\theoremstyle{definition}

\def\R{\mathbb{R}}

\def\Z{\mathbb{Z}}

\DeclareMathOperator{\Kh}{Kh}

\DeclareMathOperator{\CKh}{CKh}

\def\minus{\smallsetminus}
\def\co{\colon\thinspace}

 \DeclareMathOperator{\id}{id}

\DeclareMathOperator{\Kob}{Kob}

\definecolor{darkgreen}{rgb}{0,0.5,0}
\definecolor{purple}{rgb}{0.5,0,0.5}

\numberwithin{equation}{section}

\begin{document}

\title{Khovanov homology and ribbon concordances}

\author{Adam Simon Levine}
\address{Department of Mathematics, Duke University, Durham, NC 27708}
\email{alevine@math.duke.edu}

\author{Ian Zemke}
\address{Department of Mathematics, Princeton University, Princeton, NJ 08544}
\email{izemke@math.princeton.edu}

\thanks{ASL was partially supported by NSF grant DMS-1806437. IZ was partially supported by
NSF grant DMS-1703685. The authors are grateful to Radmila Sazdanovic for helpful conversations
about the functoriality of Khovanov homology and to the referee for thoughtful comments.}

\begin{abstract}
We show that a ribbon concordance between two links induces an injective map on Khovanov homology.
\end{abstract}

\maketitle

If $K_0$ and $K_1$ are knots in $S^3$, a \emph{concordance} from $K_0$ to $K_1$ is a smoothly
embedded annulus $C \subset S^3 \times [0,1]$ with boundary $K_0 \times \{0\} \cup K_1 \times
\{1\}$. More generally, for $n$-component links $L_0$ and $L_1$, a concordance is a disjoint union
of $n$ knot concordances between the components of $L_0$ and the components of $L_1$. Any embedded
cobordism $F \subset S^3 \times [0,1]$ is called \emph{ribbon} if projection to the $[0,1]$ factor
restricts to a Morse function on $F$ with only index $0$ and $1$ critical points. We say that $L_0$
is \emph{ribbon concordant} to $L_1$ if there exists a ribbon concordance from $L_0$ to $L_1$; note
that this is not a symmetric relation. For any cobordism $F$, let $\bar F$ denote the mirror of
$F$, viewed as a cobordism from $L_1$ to $L_0$.\footnote{Actually, Gordon defines
$F$ to be ribbon if it has only index $1$ and $2$ critical points; thus, a concordance $F$ is
ribbon in Gordon's sense iff $\bar F$ is ribbon in our sense. Our reversed convention was introduced
by the second author in \cite{ZemkeRibbon}. One justification for the change is that
we prefer to treat a ribbon disk for a knot $K$ as a cobordism from the empty link to
$K$, rather than vice versa.}

In a recent paper \cite{ZemkeRibbon}, the second author showed that knot Floer homology gives an obstruction to ribbon concordance. In this short note, we prove an analogous result for Khovanov homology \cite{KhovanovJones}.

Khovanov \cite{KhovanovJones} showed that any embedded link cobordism $F$ (not just a concordance)
in $\R^3 \times [0,1]$ between links $L_0$ and $L_1$ gives rise to a linear map
\[
\Kh(F) \co \Kh(L_0) \to \Kh(L_1).
\]
Subsequently, Khovanov \cite{KhovanovCobordism}, Jacobsson \cite{JacobssonCobordisms} and Bar-Natan
\cite{BarNatanTangles} showed that $\Kh(F)$ is invariant up to sign under isotopy of
$F$. Clark--Morrison--Walker \cite{ClarkMorrisonWalkerFunctoriality} and Caprau
\cite{CaprauCobordisms} then showed how to tweak the construction of $\Kh(F)$ so that is actually
completely invariant under isotopy of $F$, with no sign indeterminacy, and thus defines an honest
functor on the (suitably modified) cobordism category of links.\footnote{These invariance results are all proven for cobordisms in $\R^3 \times [0,1]$ rather than $S^3 \times [0,1]$. Since any cobordism in $S^3 \times
[0,1]$ can generically be assumed to miss some segment $\{s\} \times [0,1]$, we lose no generality
by considering cobordisms only in $\R^3 \times [0,1]$.} For our purposes, invariance up to sign is sufficient, so we will use the original construction of cobordism maps and not introduce the extra data needed to pin down signs. Note that when $C$ is a concordance between two knots, the map $\Kh(C)$ preserves both the homological and quantum gradings; see~\cite[\S3.4]{JacobssonCobordisms} and \cite[\S6]{BarNatanTangles}.

Our main result, which partially answers a question posed by Eisermann \cite[Question 7.7]{EisermannRibbon}, is:
\begin{theorem} \label{thm: injective}
If $C$ is a ribbon concordance from $L_0$ to $L_1$, the induced map
\[
\Kh(C) \co \Kh(L_0) \to \Kh(L_1)
\]
is injective, with left inverse given by $\Kh(\bar C)$. In particular, for any bigrading $(i,j)$, $\Kh^{i,j}(L_0)$ embeds in $\Kh^{i,j}(L_1)$ as a direct summand.
\end{theorem}
An immediate corollary is:
\begin{corollary} \label{cor: isomorphic}
If $L_0$ is ribbon concordant to $L_1$ and $L_1$ is ribbon concordant to $L_0$, then $\Kh(L_0)$ and
$\Kh(L_1)$ are isomorphic as bigraded groups.
\end{corollary}
Note that both Theorem \ref{thm: injective} and Corollary \ref{cor: isomorphic} hold with
coefficients in any ring. Corollary \ref{cor: isomorphic} provides further evidence for Gordon's
conjecture that two knots that are mutually ribbon concordant must be isotopic
\cite[Conjecture 1.1]{GordonRibbon}.

\subsection*{Applications}

Before proving Theorem \ref{thm: injective}, we state a few corollaries. For any link $L$, the
\emph{crossing number} $c(L)$ is the minimal number of crossings in any diagram of $L$. The
$\emph{Khovanov breadth}$ of $L$ is defined as $b_{\Kh}(L) := q_{\max}(L) - q_{\min}(L)$, where
$q_{\max}(L)$ and $q_{\min}(L)$ denote the maximum and minimum quantum gradings in which $\Kh(L)$ is
nonzero. For any link $L$, we have $b_{\Kh}(L) \le 2c(L)+2$, with equality if $L$ is a non-split
alternating link; this follows by adapting Murasugi's proof of the analogous statements for the
Jones polynomial \cite[Theorem~2]{MurasugiClassicalJones}. (See
\cite[Property~1.4]{AsaedaPrzytyckiTorsion}; note that the quantum grading there is twice
the one in \cite{KhovanovJones, BarNatanCategorification}.) Similarly, the \emph{Khovanov width}
$w_{\Kh}(L)$ is defined analogously to $b_{\Kh}(L)$ using the $\delta$ grading rather than the
quantum grading. We say that $L$ is \emph{$\Kh$-thin} if $w_{\Kh}(L)=2$, which is the minimum
possible value. Lee \cite{LeeKhovanov} proved that all non-split, alternating links are $\Kh$-thin;
this was extended to quasi-alternating links by Manolescu--Ozsv\'ath \cite{ManolescuOzsvathQA}.

As an immediate consequence of Theorem \ref{thm: injective}, we obtain:

\begin{corollary} \label{cor: breadth-width}
If $L_0$ is ribbon concordant to $L_1$, then $q_{\min}(L_1) \le q_{\min}(L_0)$, $q_{\max}(L_1) \ge q_{\max}(L_0)$, $b_{\Kh}(L_1) \ge b_{\Kh}(L_0)$, and $w_{\Kh}(L_1) \ge w_{\Kh}(L_0)$.
\end{corollary}

\begin{corollary} \label{cor: alternating}
If $L_0$ is ribbon concordant to $L_1$, and $L_0$ is a non-split alternating link, then $c(L_0) \le
c(L_1)$. Thus, there are only finitely many alternating links that are ribbon concordant to a given
link.
\end{corollary}

If $K_1,\dots, K_n$ are unlinked knots in $\R^3$, we say that $K'$ is a \emph{band connected sum} of $K_1,\dots, K_n$ if $K'$ is obtained by connecting $K_1,\dots, K_n$ with $n-1$ band additions (possibly in a very tangled fashion). Miyazaki \cite{MiyazakiBandSumRibbon} proved that if $K'$ is a band connected sum of $K_1,\dots, K_n$, then there is a ribbon concordance from the ordinary connected sum $K_1\# \cdots \# K_n$ to $K'$. This result, together with the inequalities stated above, immediately implies the following statements:

\begin{corollary}
If $K'$ is a band connected sum of alternating knots $K_1,\dots, K_n$, then
\[
c(K')\ge c(K_1)+\cdots+ c(K_n).
\]
\end{corollary}

\begin{corollary}
If $K'$ is a band connected sum of knots $K_1, \dots, K_n$, and $K'$ is $\Kh$-thin, then $K_1, \dots, K_n$ are $\Kh$-thin.
\end{corollary}

(See \cite[Section 1.3]{ZemkeRibbon} for further discussion of band connected sums.)

\subsection*{Proof of the main theorem}

The proof of Theorem \ref{thm: injective} follows directly from the behavior of the Khovanov
cobordism maps under two operations: disjoint union with unknotted $2$-spheres and surgery along
embedded $1$-handles. To state this result, we first recall that the Khovanov package also includes
maps associated to \emph{dotted cobordisms}, which are discussed briefly in
\cite[\S11.2]{BarNatanTangles} and then more extensively in \cite{CaprauCobordisms}. A dotted
cobordism is simply an embedded cobordism $C$ containing some finite set of marked points, which are
allowed to move around freely. The action of a dot is easy to describe on the level of chain
complexes (although we do not actually need this information for the argument below). Namely, for a
product cobordism $C = L \times I$, with a single dot, choose a marked point in a diagram for $L$
lying on the dotted component of $L$ and away from the crossings. We then obtain a chain map
$\CKh(L) \to \CKh(L)$, given by sending $v_+ \mapsto v_-$ and $v_- \mapsto 0$ on the marked
component of each resolution and extending by the identity on all other components.

The key properties that we need are the following:

\begin{proposition} \label{prop: relations}
Let $F \subset \R^3 \times [0,1]$ be an embedded cobordism from $L_0$ to $L_1$, possibly with dots.
\begin{enumerate}
\item \label{item: add-sphere} Suppose $S \subset \R^3 \times [0,1]$ is an unknotted $2$-sphere that is unlinked from $F$, and let $\dot S$ denote $S$ equipped with a dot. Then $\Kh(F \cup S)=0$ and $\Kh(F \cup \dot S) = \pm \Kh(F)$.

\item \label{item: 1-handle} Suppose $h$ is an embedded $3$-dimensional $1$-handle with ends on $F$ (and otherwise disjoint from $F$). Let $F'$ be obtained from $F$ by surgery along $h$, and let $\dot F_1$ and $\dot F_2$ be obtained by adding a dot to $F$ at either of the feet of $h$. Then $\Kh(F') = \pm \Kh(\dot F_1) \pm  \Kh(\dot F_2)$.
\end{enumerate}
\end{proposition}

\begin{figure}
\labellist
 \pinlabel $=0$ [l] at 58 88
 \pinlabel $=1$ [l] at 158 88
 \pinlabel $=0$ [l] at 300 88
 \pinlabel $=$ at 144 21
 \pinlabel $+$ at 230 21
\endlabellist
\includegraphics{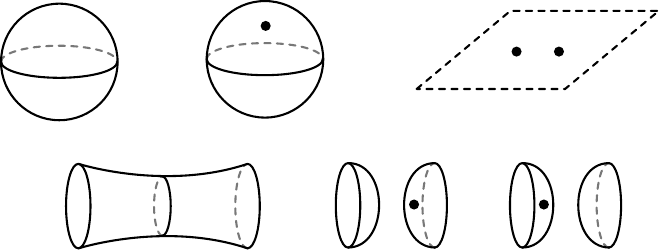}
\caption{Bar-Natan's local relations for dotted cobordisms.}\label{fig::1}
\end{figure}
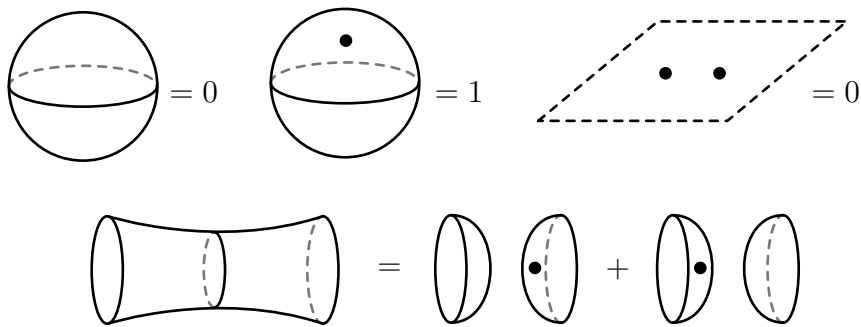

\begin{proof}
Both of these properties follow from the ``local relations'' shown in Figure \ref{fig::1}, given in \cite[\S11.2]{BarNatanTangles} and \cite[\S2.2]{CaprauCobordisms}. To be precise, Bar-Natan showed that Khovanov's construction can really be thought of as taking values in a certain abelian category $\Kob_\bullet$, whose objects are ``formal chain complexes'' of closed, embedded $1$-manifolds in the plane, and whose morphisms are ``formal chain maps'' of dotted cobordisms in $\R^2 \times [0,1]$, considered up to boundary-preserving isotopies, and modulo the local relations. Applying Khovanov's $(1+1)$ dimensional TQFT then provides a functor from $\Kob_\bullet$ to the category of chain complexes over $\Z$, and the composition agrees with the original construction of Khovanov homology.

To prove \eqref{item: add-sphere}, we may perform an ambient isotopy of $\R^3 \times [0,1]$ so that the sphere $S$ lies in a $3$-dimensional slice $\R^3 \times \{t\}$ for some $t \in (0,1)$. The first two local relations in Figure \ref{fig::1} then indicate that that the morphism (in $\Kob_\bullet$) associated to $F \cup S$ is $0$, and that the morphism associated to $F \cup \dot S$ is the same as that associated to $F$, up to a sign. After applying the TQFT, this statement then translates to the corresponding statement for the actual maps of Khovanov homology groups. Likewise, to prove \eqref{item: 1-handle}, we perform an ambient isotopy so that $h$ lies within a small ball in a slice $\R \times \{t\}$, and the intersections of $F'$, $\dot F_1$, and $\dot F_2$ with this ball can be identified with the three pictures in the second row of Figure \ref{fig::1}. The morphisms in $\Kob_\bullet$ associated to the three cobordisms, each taken with some choice of signs, then satisfy the stated relation, so the maps on Khovanov homology groups do as well.
\end{proof}

As a consequence of Proposition \ref{prop: relations}, we have:


\begin{proposition} \label{prop: attach-spheres}
Let $F \subset \R^3$ be an embedded cobordism from $L_0$ to $L_1$, possibly with dots. Let $S_1, \dots, S_k$ are disjoint, unknotted $2$-spheres that are unlinked from each other and from $F$, and let $h_1, \dots, h_k$ be disjointly embedded $1$-handles such that $h_i$ has one end on $F$ and one end on $S_i$. Let $F'$ be obtained from $F \cup S$ by surgery along $h$. Then $\Kh(F') = \pm \Kh(F)$.
\end{proposition}

\begin{proof}
For each $e = (e_1, \dots, e_k) \in \{0,1\}^k$, let $F_e$ denote the cobordism $F \cup S_1 \cup \dots \cup S_k$, with a dot on $S_i$ if and only if $e_i=1$. By induction using Proposition \ref{prop: relations}\eqref{item: 1-handle}, we have
\[
\Kh(F) = \sum_{e \in \{0,1\}^k} \pm \Kh(F_e)
\]
for some choices of signs. By Proposition \ref{prop: relations}\eqref{item: add-sphere}, we have $\Kh(F_e) = 0$ for every $e \ne (1,\dots,1)$, and $\Kh(F_{(1,\dots,1)}) = \pm \Kh(F)$.
\end{proof}

\begin{proof}[Proof of Theorem \ref{thm: injective}]
Let $C$ be a ribbon concordance from $L_0$ to $L_1$, and consider the reverse cobordism $\bar C$ from $L_1$ to $L_0$. Let $D$ denote the composite cobordism $\bar C \circ C$, which is a concordance from $C$ to itself. By the functoriality of Khovanov homology, $\Kh(D) = \Kh(\bar C) \circ \Kh(C)$. The second author showed in \cite{ZemkeRibbon} that $D$ has the following nice topological description: There exist unknotted, unlinked $2$-spheres $S_1, \dots S_n \subset (\R^3 \minus L_0) \times [0,1]$, and disjointly embedded $3$-dimensional $1$-handles $h_1, \dots, h_n$ in $\R^3 \times I$, where $h_i$ joins $L_0 \times [0,1]$ to $S_i$ and is disjoint from $S_j$ for $j \ne i$, such that $D$ is obtained from $(L_0 \times [0,1]) \cup S_1 \cup \dots \cup S_n$ by embedded surgery along the handles $h_1, \dots, h_n$. By Proposition \ref{prop: attach-spheres}, we have $\Kh(D) = \pm \Kh(L_0 \times [0,1]) = \pm \id_{\Kh(L_0)}$. It follows that $\Kh(C)$ is injective, with left inverse given by $\pm \Kh(\bar C)$.
\end{proof}

\bibliography{bibliography}
\bibliographystyle{amsalpha}

\end{document}